\documentclass[12pt, reqno]{amsart}

\usepackage[english]{babel}
\usepackage[latin1]{inputenc}
\usepackage{babel}
\usepackage{fancyhdr}
\usepackage{amsmath,amsfonts,amssymb,amsthm}
\usepackage[mathcal]{euscript}
\usepackage{mathrsfs}
\usepackage[tmargin=1.5in,bmargin=1.5in,lmargin=1.3in,rmargin=1.3in]{geometry}
\usepackage[all]{xy}
\usepackage{textcomp}
\usepackage{epsfig}
\usepackage{graphics}
\usepackage{graphicx}
\usepackage{mathrsfs}

\usepackage{caption}
\usepackage[OT2,OT1]{fontenc}
\newtheorem{thm}{Theorem}[section]

\newtheorem{cond}[thm]{Condition}

\newtheorem{lem}[thm]{Lemma}
\newtheorem{prop}[thm]{Proposition}
\newtheorem{remark}[thm]{Remark}

\usepackage{microtype}
\usepackage{bbm}
\usepackage{color}
\usepackage[textsize=small]{todonotes}       

\definecolor{halfgray}{gray}{0.55}
\definecolor{webgreen}{rgb}{0,.5,0}
\definecolor{webbrown}{rgb}{.6,0,0}
\definecolor{Maroon}{cmyk}{0, 0.87, 0.68, 0.32}
\definecolor{royalblue}{cmyk}{1, 0.50, 0, 0}
\definecolor{Black}{cmyk}{0, 0, 0, 0}

\usepackage{hyperref}
\hypersetup{%
    colorlinks=true, linktocpage=true, pdfstartpage=1, pdfstartview=FitV,%
    breaklinks=true, pdfpagemode=UseNone, pageanchor=true, pdfpagemode=UseOutlines,%
    plainpages=false, bookmarksnumbered, bookmarksopen=true, bookmarksopenlevel=1,%
    hypertexnames=true, pdfhighlight=/O,
    urlcolor=webbrown, linkcolor=royalblue, citecolor=webgreen, 
}

\numberwithin{equation}{section}

\newcommand{\sss}{\scriptscriptstyle}

\renewcommand{\P}{\mathbb{P}}

\newcommand{\indic}{\mathbbm{1}}
\newcommand{\indi}{I}

\newcommand{\Acal}{\mathcal{A}}
\newcommand{\Ncal}{\mathcal{N}}

\newcommand{\Ecal}{\mathscr{E}}
\newcommand{\Lscr}{\mathscr{L}}

\newcommand{\Cscr}{\mathscr{C}}
\newcommand{\Nscr}{\mathscr{N}}

\newcommand{\Dcal}{\mathcal{D}}

\newcommand{\Cmax}{\mathscr{C}_{\rm max}}
\newcommand{\cluster}{\mathscr{C}}
\newcommand{\Poi}{{\sf Poi}}
\newcommand{\E}{\mathbb{E}}
\newcommand{\e}{\mathrm{e}}

\newcommand{\CMd}{\mathrm{CM}_n(\boldsymbol{d})}
\renewcommand{\P}{\mathbb{P}}
\newcommand{\prob}{\mathbb{P}}
\newcommand{\op}{o_{\sss \prob}}
\newcommand{\Op}{O_{\sss \prob}}
\newcommand{\whp}{whp{}}
\newcommand{\convd}{\overset{d}{\rightarrow}}
\newcommand{\convp}{\overset{\sss\prob}{\rightarrow}}

\newcommand{\eqn}[1]{\begin{equation}#1\end{equation}}
\newcommand{\eqan}[1]{\begin{align}#1\end{align}}
\newcommand{\nn}{\nonumber}

\newcommand{\ddsum}{\sideset{_{}^{}}{_{}^{*}}\sum}

\begin{document}

\title[Critical window for connectivity in the Configuration Model]{Critical window for connectivity in the Configuration Model}

\author{Lorenzo Federico}

\author{Remco van der Hofstad}

\email{l.federico@tue.nl, r.w.v.d.hofstad@tue.nl}
\address{Department of Mathematics and Computer Science, Eindhoven University of Technology, PO Box 513, 5600 MB Eindhoven, The Netherlands}
\begin{abstract}
We identify the asymptotic probability of a configuration model $\CMd$ to produce a {\em connected} graph within its critical window for connectivity that is identified by the number of vertices of degree 1 and 2, as well as the expected degree. In this window, the probability that the graph is connected converges to a non-trivial value, and the size of the complement of the giant component weakly converges to a finite random variable. Under a finite second moment condition we also derive the asymptotics of the connectivity probability conditioned on simplicity, from which the asymptotic number of simple connected graphs with a prescribed degree sequence follows. 
\medskip	

\end{abstract}
\date{\today}
\maketitle

\section{Introduction}

In this paper we consider the configuration model $\CMd$ on $n$ vertices with a prescribed degree sequence $\boldsymbol d = (d_1, d_2,...,d_n)$. We identify the probability that $\CMd$ is connected in terms of $\boldsymbol d$ in the limit as $n \to \infty$. We further analyse the behaviour of the model in the critical window for connectivity (i.e., when the asymptotic probability of producing a connected graph is in the interval $(0,1)$). Given a vertex $v\in[n]:=\{1,2,...,n\},$ we call $d_v$ its degree. The configuration model is constructed by assigning $d_v$ half-edges to each vertex $v$, after which the half-edges are paired randomly: first we pick two half-edges at random and create an edge out of them, then we pick two half-edges at random from the set of remaining half-edges and pair them into an edge, etc. We assume the total degree $\sum_{v\in [n]} d_v$ to be even. The construction can give rise to self-loops and multiple edges between vertices, but these imperfections are relatively rare when $n$ is large; see \cite{Hofs17, Jan09, Jan14}.

We define the random variable $D_n$ as the degree of a vertex chosen uniformly at random from the vertex set $[n]$. We call $\Nscr_i$ the set of all vertices of degree $i$ and $n_i$ its cardinality. 
The configuration model is known to have a phase transition for the existence of a giant component with critical point at
 	\[ 
	\nu_n =\frac{\E [D_n(D_n-1)]}{\E[D_n]}=1 
	\] 
 (see e.g., \cite{MolRee95} or \cite{JanLuc09}). When $\nu_n\rightarrow \nu>1$, there is a (unique) giant component $\Cmax$ containing a positive proportion of the vertices, while when $\nu_n\rightarrow \nu\leq 1$, the maximal connected component contains a vanishing proportion of the vertices. Assuming that the second moment of $D_n$ remains uniformly bounded, the subcritical behaviour was analysed by Janson in \cite{Jan08}. 
 
In this paper, we focus on the {\em connectivity transition} of the configuration model. Let us first describe the history of this problem. Wormald \cite{Wor81} showed that for $k \geq 3$ a random $k$-regular graph on $n$ vertices is with high probability $k$-connected as $n \to \infty$ (see also \cite{Boll01}). Tomasz \L uczak \cite{Luc92} proved that also if the graph is not regular, but $d_v \geq k$ for every $v \in [n]$, then $\CMd$ in with high probability $k$-connected, and found the asymptotic probability to have a connected graph when $d_v \geq 2$ and the graph is simple. Actually \L uczak's model was defined in a different way from the configuration model and does not allow for vertices of degree $1$, but the results could easily be adapted to the configuration model. 

In this paper we extend \L uczak's results to all possible kinds of weakly converging degree distributions, allowing vertices of degree 1 and removing any condition on maximum degree or finiteness of moments of the degree distribution. This implies that our results include heavy-tailed degree distribution that has received considerable attention in research on the configuration model recently. Moreover, using the multivariate method of moments, we identify the limiting distribution of the size of the {\em complement} of the maximal connected component $[n] \setminus \Cmax$ as a mixture of Poisson variables.

We start by introducing some notation.
\medskip

\paragraph{\bf Notation.}
All limits in this paper are taken as $n$ tends to infinity unless stated otherwise.
A sequence of events $(\mathcal{A}_n)_{n \geq 1}$ happens \emph{with high probability (\whp)} if $\P(\mathcal{A}_n) \to 1$.
For random variables $(X_n)_{n \geq 1}, X$, we write $X_n \convd X$ and $X_n \convp X$ to denote convergence in distribution and in probability, respectively.
For real-valued sequences $(a_n)_{n \geq 1}$, $(b_n)_{n \geq 1}$, we write $a_n=O(b_n)$ if the sequence $(a_n/b_n)_{n \geq 1}$ is bounded; $a_n=o(b_n)$ if $a_n/b_n \to 0$; $a_n =\Theta(b_n)$ if the sequences $(a_n/b_n)_{n \geq 1}$ and $(b_n/a_n)_{n \geq 1}$ are both bounded; and $a_n \sim b_n$ if $a_n/b_n \to 1$.
Similarly, for sequences $(X_n)_{n \geq 1}$, $(Y_n)_{n \geq 1}$ of random variables, we write $X_n=\Op(Y_n)$ if the sequence $(X_n/Y_n)_{n \geq 1}$ is tight; and $X_n=\op(Y_n)$ if $X_n/ Y_n \convp 0$. Moreover, $\Poi(\lambda)$ always denotes a Poisson distributed random variable with mean $\lambda$ and ${\sf Bin} (n,p)$ denotes a random variable with binomial distribution with parameters $n$ and $p$.

\section{Main Results}
We start by defining the conditions for $\CMd$ to be in the connectivity critical window. We define the random variable $D_n$ as the degree of a vertex chosen uniformly at random in $[n]$. We state the conditions we assume to hold throughout this paper.

\begin{cond}[Critical window for connectivity]
\label{cond-crit-conn}
We define a sequence $\CMd$ to be in the \emph{critical window for connectivity} when the following conditions are satisfied:
\begin{enumerate}
\item There exists a limiting degree variable $D$ such that $D_n \overset{d}{\to} D$;
\item $n_0 =0$;
\item $\lim_{n \to \infty} n_1/\sqrt{n}= \rho_1 \in [0 , \infty)$;
\item $\lim_{n \to \infty} n_2/n = p_2 \in [0,1)$;
\item $\lim_{n \to \infty} \E [D_n]= d < \infty$.
\end{enumerate}
\end{cond}
\medskip

Notice that we use different symbols for $\rho_1$ and $p_2$ to stress the fact that $p_2$ is actually the limit probability that a uniformly chosen vertex has degree $2$ while $\rho_1$ is obtained through a rescaling (in the critical window for connectivity $p_1 = \lim n_1/n =0$). Under these conditions, we prove our main theorem. In its statement, we write $\Cmax$ for the maximal connected component in $\CMd$.

\begin{thm}[Connectivity threshold for the configuration model]\label{main}
Consider $\CMd$ in the critical window for connectivity as described in Condition \ref{cond-crit-conn}. Then
	\begin{equation}\label{conn}
	\lim_{n \to \infty} \P (\CMd \text {\emph{ is connected}}) =\left(\dfrac{d-2p_2}{d}\right)^{1/2} \exp \left( - \dfrac{\rho_1 ^2 }{2(d -2p_2)} \right).
	\end{equation}
Moreover,
	\begin{equation}\label{dist}
	n-|\Cmax | \overset{d}{\to} X,
	\end{equation} 
where $X = \sum_k k(C_k + L_k)$, and $(C_k , L_k)_{k\geq 1}$ are independent random variables such that
	\begin{equation*}
	L_k \overset{d}{=} \Poi \left( \dfrac{\rho_1^{2} (2p_2)^{k-2}}{2d^{k-1}}\right), \quad C_k  \overset{d}{=} 
	\Poi \left( \dfrac{(2p_2)^{k}}{2k d^{k}}\right).
	\end{equation*}
Finally,\footnote{In the published version of the article, equation \eqref{expected} had a wrong value for the expected number of vertices in line components in the first version, as a consequence of a computational mistake in \eqref{eq-exp}, which has now been corrected.}
	\begin{equation}\label{expected}
	\lim_{n \to \infty} \E [n - |\Cmax|] = \dfrac{\rho_1^2}{d-p_2}+\dfrac{p_2}{d-2p_2}.
	\end{equation}
\end{thm}
\medskip

The convergence in distribution of $n-|\Cmax |$ to a proper random variable with finite mean is a stronger result than proved by \L uczak \cite{Luc92}, who instead proved that
	\begin{equation}
	\dfrac{|\Cmax|}{n}\convp 1.
	\end{equation}
Our improvement is achieved by an application of the multivariate method of moments, as well as a careful estimate of the probability that there exists $v \in [n]$ with $d_v\geq 3$ that is not part of $|\Cmax|$. We next investigate the boundary cases.

\begin{remark}[Boundary cases]
\label{rem-bc}
Our proof also applies to the boundary cases where $\rho_1 = \infty, p_2=0$ or $d=\infty$. When $d<\infty$, we obtain
	\begin{equation}
	\label{bc-1}
	\P(\CMd \text {\emph{ is connected}}) \to  \begin{cases} 
	0 & \text{ when } \rho_1 = \infty,\\
	1 & \text{ when } \rho_1, p_2 = 0.
	\end{cases}
	\end{equation}
When $d = \infty$, instead
	\begin{equation}
	\label{bc-2}
	\lim_{n \to \infty} \P (\CMd \text{ is connected})= \lim_{n \to \infty} \exp\Big(-\frac{n_1^2}{2 \ell_n}\Big),
	\end{equation}
	where $\ell_n = \sum_{i \in [n]} d_i$ denotes the total degree.
\end{remark}
\medskip

Assuming also that $D$ has finite second moment the configuration model is simple with non-vanishing probability.
Under this extra assumption, the next theorem states how many connected graphs there are with prescribed degrees in the connectivity window defined in Condition \ref{cond-crit-conn}.  

\begin{thm}[Connectivity conditioned on simplicity and number of connected simple graphs]\label{mainsim}
Consider $\CMd$ in the connectivity critical window defined in Condition \ref{cond-crit-conn}.  If 
	\[
	\lim_{n \to \infty}\dfrac{\E [D_n(D_n-1)]}{\E[D_n]}= \nu < \infty,
	\]
then
	\begin{equation}
	\label{conns}\begin{split}
	&\lim_{n \to \infty}\P(\CMd \text {\emph{ is connected}}
	\mid\CMd \text {\emph{ is simple}}) \\
	&\qquad\qquad=\left(\dfrac{d-2p_2}{d}\right)^{1/2}\exp{\left(-\dfrac{\rho_1^2}{2(d-2p_2)} +\dfrac{p_2^2+d p_2}{d^2} \right)}.
	\end{split}\end{equation}

Let $\mathscr{N}^C_n(\boldsymbol{d})$ be the number of connected simple graphs with degree distribution $\boldsymbol{d}$. Then
	\begin{equation}
	\begin{split}
	\mathscr{N}^C_n(\boldsymbol{d}) = & \dfrac{( \ell_n-1)!!}{\prod_{i \in [n]}d_i !}\left(\dfrac{d-2p_2}{d}\right)^{1/2} \\
	&\quad \times\exp\left( -\dfrac{\nu}{2} -\dfrac{\nu^2}{4}-\dfrac{\rho_1^2}{2(d-2p_2)} 
	+\dfrac{p_2^2+d p_2}{d^2}\right) (1+o(1)).
	\end{split}
	\end{equation}

\end{thm}
\medskip

It is instructive to compare the asymptotics of the number of connected simple graphs with degree distribution $\boldsymbol{d}$ to the number $\mathscr{N}_n(\boldsymbol{d})$ of simple graphs with a given degree sequence $\boldsymbol{d}$, as identified by Janson \cite{Jan09},
\begin{equation}
\mathscr{N}_n(\boldsymbol{d})= \dfrac{( \ell_n-1)!!}{\prod_{i \in [n]}d_i !}\exp\left( -\dfrac{\nu}{2} -\dfrac{\nu^2}{4}\right) (1+o(1)).
\end{equation}

With these results, the connectivity critical window is fully explored. Indeed, we have determined the asymptotic probability for the configuration model to produce a connected graph for all possible choices of the limiting degree distribution. What remains is to find the asymptotic of the number of connected simple graphs with degree distribution $\boldsymbol d$ when it is below the connectivity critical window (i.e., when $n_1 \gg n^{1/2}$). In this case we should analyse how fast the probability to produce a connected graph vanishes, which is a hard problem.

To address the last boundary case, not considered in Remark \ref{rem-bc}, it is also worth noticing that the size of the largest component is very sensitive to the precise way how $n_2/n \to 1$
(recall that we assume that $p_2<1$ in Condition \ref{cond-crit-conn}), as we describe now. 
We define $\cluster (v)$ as the connected component of a uniformly chosen vertex. When $n_2=n$, it is not hard to see that
	\begin{equation}
	\dfrac{|\Cmax|}{n} \overset{d}{\to} S; \qquad\quad \dfrac{|\cluster(v)|}{n} \overset{d}{\to} T,
	\end{equation}
where $S,T$ are proper random variables that satisfy the relation $S  \overset{d}{=} T \vee [(1-T)S]$.
Instead, $|\Cmax|/n\convp 0$ when $n_2 = n - n_1$, with $n_1 \to \infty$, while $|\Cmax|/n\convp 1$ when $n_2 = n - n_4$, with $n_4 \to \infty$. The latter two statements can be proved by relating it to the case where $n_2=n$. Indeed, for the case where $n_1>0$, we take $n_2'=n_2+n_1/2$, and produce $\CMd$ from the configuration model with $n_2$ vertices of degree 2 by `splitting' $n_1/2$ vertices of degree 2 into two vertices of degree 1. For the case where $n_4>0$, we take $n_2'=n_2+2n_4$, and produce $\CMd$ from the configuration model with $n_2$ vertices of degree 2 by `merging' $2n_4$ vertices of degree 2 into a vertex of degree 4. This explains why it is important to us to assume that $p_2<1$ in Condition \ref{cond-crit-conn}.

\subsection{Outline of the proof}

We first notice that in the connectivity critical window our configuration model is supercritical, i.e., \whp \  it has a unique component of linear size with respect to the whole graph. In more detail, for finite $\rho_1 < \infty$ and $ p_2< 1$,
	\begin{equation}
	\lim_{n \to \infty} \nu_n= \lim_{n \to \infty}\dfrac{\E [D_n(D_n-1)]}{\E[D_n]}
	\geq \dfrac{2p_2 + 6 (1-p_2 )}{2 p_2 +3 (1- p_2 )} 
	> 1.
	\end{equation}
Thus the results from \cite{JanLuc09,MolRee98} imply that $|\Cmax|= \Theta_\P (n)$, while the second largest connected component $\cluster_{\sss (2)}$ satisfies $|\cluster_{\sss (2)}| = o_\P(n)$ and $|E(\cluster_{\sss (2)})|=o_\P(n)$, where $|E(G)|$ indicates the number of edges of the graph $G$. The proof of our main theorem is now divided into two parts:

\begin{enumerate}
\item To identify the limit distribution of the number of lines and cycles that form $[n] \setminus \Cmax$, which we do in Section \ref{sec-Pois-conv};

\item To prove that \whp~ all vertices $v \in [n]$ with $d_v \geq 3$ are in the giant component $\Cmax$, which we do in Section \ref{sec-conn-three}.
\end{enumerate}

The proofs of our main theorems are then completed in Section \ref{sec-compl-pf}.

\section{Poisson convergence of the number of lines and cycles}
\label{sec-Pois-conv}
In this section, we prove that the number of cycles (components made by $k$ vertices of degree 2) and lines (components made by 2 vertices of degree 1 and $k-2$ vertices of degree 2) jointly converge to independent Poisson random variables. In Section \ref{sec-conn-three}, we will show that $[n] \setminus \Cmax$ \whp{} only contains vertices of degree 1 and 2, so that all the other components are either cycles or lines.
We define the sequences of random variables $(\mathbf{C}_n, \mathbf{L}_n)= \big((C_k(n),L_k(n))\big)_{k\geq 1}$ as

\begin{itemize}
\item[$\rhd$] $C_k(n)$=  $\#$ \{cycles of length $k$ in $\CMd$\},
\item[$\rhd$] $L_k(n)$= $\#$ \{lines of length $k$ in $\CMd$\}.
\end{itemize}
\medskip

We consider a vertex of degree 2 with a self-loop as a cycle of length 1. By convention, $L_1(n)=0$ for all $n$ since a vertex of degree $1$ can not have a self-loop.

We define $\Cscr_k= \{ \{v_1,v_2,...,v_k\} \subseteq \Nscr_2\}$ to be the set of all collections of $k$ vertices that could form a cycle, and denote
	\begin{equation}
	C_k (n) = \sum_{c \in \Cscr_k} \indic_{\{c \text{ forms a cycle}\}},
	\end{equation}
where $\indic_{A}$ denotes the indicator of the event $A$.
In a similar way we define $\Lscr_k=\{ \{v_1,v_2,...,v_k\}: v_1, v_k \in \Nscr_1; v_2,...,v_{k-1} \in \Nscr_2 \}$ to be the set of all collections of $k$ vertices that could form a line, and denote
	\begin{equation}
	L_k (n) = \sum_{l \in \Lscr_k} \indic_{\{l \text{ forms a line}\}}.
	\end{equation} 

We will use the multivariate method of moments to show that $\big((C_n (k), \ L_k(n))\big)_{ k \geq 1}$ converges to a vector of independent Poisson random variables. For a random variable $X$, we define $(X)_r=X (X -1)\cdots(X -r+1)$. For the multivariate method of moments, we recall two useful lemmas, whose proofs are given in \cite[Section 2.1]{Hofs17}.

\begin{lem}[Multivariate moment method with Poisson limit]
\label{poi} 
A sequence of vectors of non-negative integer-valued random variables ~$(X_1^{\sss(n)}, X_2^{\sss(n)},...,X_k^{\sss(n)} )_{n\geq 1}$ converges in distribution to a vector of independent Poisson random variables
with parameters $(\lambda_1 , \lambda_2,...,\lambda_k)$ when, for all possible choices of $(r_1, r_2,...,r_k) \in \mathbb{N}^{k} $,
	\begin{equation}
	\lim_{n\to \infty} \E [(X_1^{\sss(n)})_{r_1}(X_2^{\sss(n)})_{r_2}\cdots (X_k^{\sss(n)})_{r_k}] 
	= \lambda_1 ^{r_1} \lambda_2^{r_2} \cdots \lambda_k^{r_k}.
	\end{equation}
\end{lem}

\begin{lem}[Factorial moments of sums of indicators]
\label{ind}When $X_j=\sum_{i \in \mathcal{I}_j}\indi_{i^{(j)}}$ for all $j =1,\ldots,k$,
	\begin{equation}
	\begin{split}
	&\E[(X_1^{\sss(n)})_{r_1}(X_2^{\sss(n)})_{r_2}\cdots(X_k^{\sss(n)})_{r_k}]\\ =&\ddsum_{i_1^{(1)},\ldots,i_{r_1}^{(1)}\in \mathcal{I}_1} 
	\cdots \ddsum_{i_1^{(k)},\ldots,i_{r_k}^{(k)}\in \mathcal{I}_k} \E\Big[\prod_{j=1}^k\prod_{s=1}^{r_k}\indi_{i_s^{(j)}}\Big ],
	\end{split}
	\end{equation}
where $\sum^{\ast}$ denotes a sum over distinct indices. 
\end{lem}
\medskip

See also \cite[Chapter 6]{JanLucRuc00} for more general versions of the method of moments.
We now can state the main result of this section.

\begin{thm}[Poisson convergence of number of lines and cycles] 
\label{poicl}
Consider $\CMd$ in the critical window for connectivity defined in Condition \ref{cond-crit-conn}. Then
	\begin{equation}\label{conver}
	(\mathbf{C}_n, \mathbf{L}_n)  \xrightarrow{d} (\mathbf{C}, \mathbf{L}),
	\end{equation}
where $(\mathbf{C}, \mathbf{L})=\big((C_k,L_k)\big)_{ k \geq 1}$ is a sequence of independent random variables with
	\begin{equation}
	L_k \overset{d}{=} \Poi \left( \dfrac{\rho_1^{2} (2p_2)^{k-2}}{2d^{k-1}}\right), \quad C_k  \overset{d}{=} 
	\Poi \left( \dfrac{(2p_2)^{k}}{2k d^{k}}\right),
	\end{equation}
	and the convergence in \eqref{conver} is in the product topology on $\mathbb N^\infty$.
\end{thm}
\medskip

\proof We want to find the combined factorial moments of $(L_j(n), C_j(n))_{j \leq k}$ and show that
	\eqan{
	\label{mome}
	&\E [(C_1(n))_{r_1}(L_2(n))_{s_2}\cdots (C_k(n))_{r_k}(L_k(n))_{s_k}]\\
	&\qquad\to \prod_{j=2}^{k}  \left( \dfrac{\rho_1^{2} (2p_2)^{j-2}}{d^{j-1}}\right)^{s_j} 
	\prod_{j=1}^{k} \left( \dfrac{(2p_2)^{j}}{2kd^{j}}\right)^{r_j}.\nn
	}
We argue by induction on $k$. When $k=0$, both sides in \eqref{mome} are equal to 1, which initializes the induction hypothesis.

We next argue how to advance the induction hypothesis. We define 
	\[
	w_{k,j}(\mathbf r,\mathbf s)= \{ c_i (1),\ldots, c_i(r_i)\in \Cscr_i, 1 \leq i \leq k;  l_i(1),\ldots, l_{i} (s_i) \in \Lscr_i, 2\leq i \leq j \},
	\]
where all $c_i (1),\ldots, c_i(r_i)$ and $l_i(1),\ldots, l_{i} (s_i)$ are ordered lists without repetitions.
Further, $\mathscr{E}(w_{k,j}(\mathbf r,\mathbf s))$ denotes the event that all $c_i(h) \in w_{k,j}(\mathbf r,\mathbf s)$ form a cycle and all $l_i(h) \in w_{k,j}(\mathbf r,\mathbf s)$ form a line.
By Lemma \ref{ind},
	\begin{equation}
	\E [(C_1(n))_{r_1}(L_2(n))_{s_2}\cdots (C_k(n))_{r_k}(L_k(n))_{s_k}]
	=\sum_{w_{k,k}(\mathbf r,\mathbf s) }\P (\mathscr{E}(w_{k,k}(\mathbf r,\mathbf s))).
	\end{equation}
We rewrite this as 
	\begin{equation}
	\sum_{w_{k,k-1}(\mathbf r,\mathbf s) }\P (\mathscr{E}(w_{k,k-1}(\mathbf r,\mathbf s))) \ddsum_{l_1,\ldots,l_{s_k}\in \Lscr_k}
	\E[\indi_{i_1} \indi_{i_2} \cdots \indi_{i_{s_k}}\mid \mathscr{E}(w_{k,k-1}(\mathbf r,\mathbf s))],
	\end{equation}
where $\indi_{i_{s}}$ is the indicator that the vertices in $c_{i_s}$ form a line.

We call $a_1$ and $a_2$ the number of vertices of degree 1 and 2 necessary to create the cycles and lines prescribed by $w_{k,k-1}(\mathbf r,\mathbf s)$ and $a_e = a_1 +2a_2$ the number of half-edges they have. The values of $a_1, a_2$  are completely independent from the exact choice of $w_{k,k-1}(\mathbf r,\mathbf s)$ as long as all sets are disjoint (otherwise the event $\mathscr{E}(w_{k,k-1}(\mathbf r,\mathbf s))$ is impossible).  The number of possible choices of $s_k$ different disjoint $l \in \Lscr_k$ without using the vertices allocated for $w_{k,k-1}(\mathbf r,\mathbf s)$ are
	\begin{equation}\begin{split}\label{linnum}
	\dfrac{(n_1-a_1)!}{2^{s_k} (n_1-a_1 -2s_k)!}&\dfrac{(n_2 - a_2)!}{(k-2)!^{s_k}(n_2 - a_2-(k-2)s_k)!}(1+o(1))\\ 
	= &	\dfrac{n_1^{2s_k}}{2^{s_k}}\dfrac{n_2^{(k-2)s_k}}{(k-2)!^{s_k}}(1+o(1)),
	\end{split}
	\end{equation}
provided that $n_1, n_2 \to \infty$.
The probability that the first forms a line is
	\eqan{
	&\dfrac{2k-4}{\ell_n-a_e-1}\dfrac{2k-6}{\ell_n-a_e-3}\cdots \dfrac{2}{\ell_n-a_e-2k+5} 
	\dfrac{1}{\ell_n-a_e-2k+3}\nn\\
	&\qquad=\dfrac{(2k-4)!!}{\ell_n^{k-1}}(1+o(1)).
	}
For all the other lines we just have to subtract from $\ell_n-a_e$ the $2k-2$ half-edges that we have used for each of the previous ones, so that
	\begin{equation}
	\E[\indi_{i_1} \indi_{i_2} \cdots \indi_{i_{s_k}}\vert \mathscr{E}(w_{k,k-1}(\mathbf r,\mathbf s))]
	= \dfrac{(2k-4)!!^{s_k}}{\ell_n^{s_k(k-1)}}(1+o(1)).
	\end{equation}
Finally we obtain
	\begin{equation}\begin{split}
	\ddsum_{l_1,\ldots,l_{s_k}\in \Lscr_k}
	& \E[\indi_{i_1} \indi_{i_2} \cdots \indi_{i_{s_k}}\vert \mathscr{E}(w_{k,k-1}(\mathbf r,\mathbf s))]\\ 
	&= \dfrac{(\rho_1^{2}n)^{s_k}}{2^{s_k}}\dfrac{(p_2 n)^{(k-2)s_k}}{(k-2)!^{s_k}}
	\dfrac{(2k-4)!!^{s_k}}{\ell_n^{s_k(k-1)}}(1+o(1))
	\\&=\left( \dfrac{\rho_1^{2}(2n_2)^{k-2}}{2d^{k-1}}\right)^{s_k}(1+o(1)).
	\end{split}\end{equation}

We do the same for the cycles $C_k(n)$, writing
	\begin{equation}
	\sum_{w_{k-1,k-1}(\mathbf r,\mathbf s) }\P (\mathscr{E}(w_{k-1,k-1}(\mathbf r,\mathbf s))) 
	\ddsum_{c_1,\ldots,c_{r_k}\in \Cscr_k} \E[\indi_{i_1} \cdots \indi_{i_{r_k}}
	\vert \mathscr{E}(w_{k-1,k-1}(\mathbf r,\mathbf s))].
	\end{equation}
The number of possible choices of $r_k$ different disjoint $c \in \Cscr_k$ without using the vertices allocated for $w_{k-1,k-1}(\mathbf r,\mathbf s)$ are
	\begin{equation}\label{cicnum}
	\dfrac{(n_2 - a_2)!}{k!^{r_k}(n_2 - a_2-kr_k)!}(1+o(1))= \dfrac{(n_2)^{kr_k}}{k!^{r_k}}(1+o(1)),
	\end{equation}
provided that $n_2 \to \infty$.

The probability that the first forms a cycle is
	\begin{equation}
	\dfrac{2k-2}{\ell_n-a_e-3}\dfrac{2k-4}{\ell_n-a_e-5}\cdots \dfrac{2}{\ell_n-a_e-2k+3} \dfrac{1}{\ell_n-a_e-2k+1}(1+o(1)).
	\end{equation}
Again, for all the other cycles we just have to subtract the $2k$ half-edges that we have used for the previous ones so that
	\begin{equation}
	\E[\indi_{i_1} \indi_{i_2} \cdots \indi_{i_{r_k}}\vert \mathscr{E}(w_{k-1,k-1}(\mathbf r,\mathbf s))]
	= \dfrac{(2k-2)!!^{r_k}}{\ell_n^{r_k k}}(1+o(1)).
	\end{equation}
Thus, we obtain
	\begin{equation}\begin{split}
	\ddsum_{c_1,\ldots,c_{r_k}\in \Cscr_k}
	& \E[\indi_{i_1} \indi_{i_2} \cdots \indi_{i_{s_k}}\vert \mathscr{E}(w_{k-1,k-1}(\mathbf r,\mathbf s))]\\
	=&\dfrac{n_2^{kr_k}}{k!^{r_k}}\dfrac{(2k-2)!!^{r_k}}{(\ell_n)^{r_kk}}(1+o(1))
	=\left( \dfrac{(2p_2)^k}{2kd^k}\right)^{r_k}(1+o(1)).
	\end{split}\end{equation}
This advances the induction hypothesis. We now use induction to show that \eqref{mome} holds for every $k\geq 0$, and consequently prove the claim through the method of moments
in Lemma \ref{poi}.

In \eqref{linnum} and \eqref{cicnum}, we have assumed that $n_1, n_2 \to \infty$. When this is not satisfied we can just use a first moment method to show that $L_k(n) \overset{\P}{\to} 0$ for all $k$ if $n_1 = O(1)$, and $C_k(n)  \overset{\P}{\to} 0$ for all $k$ and $L_k(n)  \overset{\P}{\to} 0$ for $k\geq 3$ if $n_2 = O(1)$.
 \qed
\bigskip

We next show that in case of finite second moment, in particular, under the condition
	\[
	\lim_{n \to \infty} \dfrac{\E[D_n(D_n -1)]}{\E[D_n]} \to \nu < \infty,
	\] 
asymptotic distribution of the number of self-loops and multiple edges is independent from $(C_k)_{k\geq 3}$ and $(L_k)_{k \geq 2}$.

We first notice that connectivity and simplicity are {\em not} independent, since self-loops and multiple edges among vertices of degree 2 make the graph simultaneously disconnected and not simple, so for $\CMd$ to be simple, we have to require $C_1 (n) = C_2(n)=0$.

We define the number of self-loops and multiple edges in $\CMd$ by $S(n), M(n)$ and show the following joint convergence.

\begin{thm}[Poisson convergence of number self-loops and multiple edges] 
\label{poism}
Consider $\CMd$ in the critical window for connectivity defined in Condition \ref{cond-crit-conn}, and let $\nu_n= \E [D_n(D_n-1)]/\E[D_n] \to \nu \leq \infty$. Then
	\begin{equation}
	((L_k(n))_{ k \geq 2},(C_k(n))_{k\geq 3},S(n),M(n))  \xrightarrow{d} ((L_k)_{ k \geq 2},(C_k)_{ k \geq 3},S,M),
	\end{equation}
with $((L_k)_{ k \geq 2},(C_k)_{ k \geq 3},S,M)$ independent Poisson random variables with
	\begin{equation}\begin{split}
	L_k \overset{d}{=} \Poi \left(\dfrac{\rho_1^2 (2p_2)^{k-2}}{2d^{k-1}}\right),& \qquad
	C_k \overset{d}{=} \Poi \left(\dfrac{(2p_2)^k}{2kd^k} \right),\\
	S \overset{d}{=} \Poi \left(\nu/2\right),&\qquad
	M \overset{d}{=} \Poi \left(\nu^2/4\right).\quad
	\end{split}\end{equation}
\end{thm}

\proof We again use multivariate method of moments in Lemma \ref{poi}. We aim to find the combined factorial moments of $((L_j(n))_{2\leq j \leq k}, (C_j(n))_{3\leq j \leq k},S(n),M(n))$, and show that
	\begin{equation}
	\label{mome2}
	\begin{split}
	\E [(L_2(n))_{s_2}(C_3(n))_{r_3}\cdots (C_k(n))_{r_k}(L_k(n))_{s_k}(S(n))_{t}(M(n))_{u}] 
	\\ \to \left(\dfrac{\nu}{2}\right)^{t+2u}\prod_{j=2}^{k}  \left( \dfrac{\rho_1^{2} (2p_2)^{j-2}}{2d^{j-1}}\right)^{s_j} 
	\prod_{j=1}^{k} \left( \dfrac{(2p_2)^{j}}{2kd^{j}}\right)^{r_j}.
	\end{split}
	\end{equation}
We now define 
	\[
	w'_{k,j}(\mathbf r,\mathbf s)= \{ c_i (1), ..., c_i(r_i)\in \Cscr_i, 3 \leq i \leq k;  l_i(1),...,l_{i}(s_i) \in \Lscr_i, 2\leq i \leq j \}
	\]
as the choice of subsets that can form such lines and cycles, and by Lemma \ref{poi},
	\begin{equation}
	\sum_{w'_{k,k}(\mathbf r,\mathbf s)}\P (\Ecal(w'_{k,k}(\mathbf r,\mathbf s)))\E [(S (n))_{t}(M (n))_{u}\mid \Ecal(w'_{k,k}(\mathbf r,\mathbf s))].
	\end{equation}

Conditionally on $\Ecal(w'_{k,k}(\mathbf r,\mathbf s))$, the random vector $(S(n), M(n))$ has the same law as the number of self-loops and multiple edges in a configuration model with degree sequence $\boldsymbol{d}'$, which is obtained from $\boldsymbol{d}$ by removing the vertices appearing in $w'_{k,k}(\mathbf r,\mathbf s)$. We notice that $\boldsymbol{d}'$ is independent from the exact choice of $w'_{k,k}(\mathbf r,\mathbf s)$. Thus, when $D'_n$ denotes the degree of a uniform random vertex selected from $\boldsymbol{d}'$ and $\nu' = \lim_{n \to \infty} \dfrac{\E[D'_n(D'_n -1)]}{\E[D'_n]}$, 
	\[
	\E [(S (n))_{t} (M(n))_{u}] \to \left(\dfrac{\nu'}{2}\right)^{t+2u}
	\]
(see e.g., \cite{Jan09, Jan14}). Since we are removing only a finite number of vertices from $\boldsymbol{d}$, we have that $\nu' =\nu$ and we thus obtain
	\begin{equation}
	\left(\nu/2\right)^{t+2u}\sum_{w'_{k,k}(\mathbf r,\mathbf s)}\P (\Ecal(w'_{k,k}(\mathbf r,\mathbf s))).
	\end{equation}
We finally obtain \eqref{mome2} using the same induction argument used to prove \eqref{mome}, which completes the proof. 
\qed

\section{Connectivity among vertices of degree at least three}
\label{sec-conn-three}

In this section, we show that in the connectivity critical window \whp{} all vertices $v$ with $d_v \geq 3$ are in the giant component. This result is already known when $\min_{i \in [n]} d_i \geq 2$, we show that it still holds even in the presence of a sufficiently small amount of vertices of degree $1$ and is stated in the following theorem.

\begin{thm}[Connectivity among vertices with $d_v \geq 3$]
\label{geq3}
Consider $\CMd$ in the connectivity critical window defined in Condition \ref{cond-crit-conn}. Then 
	\begin{equation}
	\label{bd-1}
	\E [\#\{ v \in [n] \colon d_v \geq 3, |\cluster(v)| < n/2\} ] \to 0.
	\end{equation}
Consequently,
	\begin{equation}
	\label{bd-2}
	\E [\#\{ v \in [n] \setminus \Cmax \colon d_v \geq 3\}] \to 0.
	\end{equation}
\end{thm}

We will use the usual exploration process of the configuration model, as we describe now. At each time $t$, we define the sets of half-edges $\{ \Acal_t , \Dcal_t , \Ncal_t \}$ (the active, dead and neutral sets), and explore them in the following way:

\begin{itemize}
\item[{\tt Initalize}] We pick a vertex $v \in [n]$ uniformly at random with $d_v \geq 3$ and we set all its half-edges as active. All other half-edges are set as neutral.
\item[{\tt Step}] At each step $t$, we pick a half-edge $e_1(t) $ in $\Acal_t$ uniformly at random, and we pair it with another half-edge $e_2(t)$ chosen uniformly at random in $\Acal_t \cup \Ncal_t$. We set $e_1(t), e_2(t)$ as dead.\\
\textbf{If} $e_2(t) \in \Ncal_t$, then we find the vertex $v(e_2(t))$ incident to $e_2(t)$ and activate all its other half-edges.
\end{itemize}

As usual, the above exploration forms the graph at the same time as that it explores the neighborhood of the vertex $v$. A convenient way to encode the randomness in the exploration algorithm is to first choose a permutation $\xi$ of the half-edges, chosen uniformly at random from the set of all permutations of the half-edges. Then we run the exploration choosing as $e_1(t)$ and $e_2(t)$ always the first feasible half-edges in the permutation according to the exploration rules. This means that we take the first available active half-edge as $e_1(t)$, pair it to the first available active or neutral half-edge as $e_2(t)$ to create an edge consisting of $e_1(t)$ and $e_2(t)$, and then to update the status of all the half-edges as above. 

The above description, that we will rely on for the remainder of this document, offers the possibility to analyse some properties of the exploration before running it and will be useful to prove that whp we will not run out of high-degree vertices too early in the exploration.

We define the process $S_t^{\sss(v)} = |\Acal_t|$. The update rules of $S_t^{\sss(v)}$ are
	\begin{equation}
	\label{St-rec}
	S_0^{\sss(v)}= d_v, \quad\qquad 
	S_{t+1}^{\sss(v)}-S_t^{\sss(v)}= \begin{cases}d_{v(e_2(t))}-2 & \text{ if } e_2(t) \in \Ncal_t, \\
	-2 & \text{ if } e_2(t) \in \Acal_t.\end{cases}
	\end{equation}

We define $T_0$ as the smallest $t$ such that $X_t =0$ and 
	\eqn{
	\label{T-half-def}
	T_{1/2}=\max \{t\colon |\Ncal_t| > n/2\}.
	}
By definition of the exploration process, if $T_0 \geq T_{1/2}$ then $|\cluster (v)| \geq n/2$ (and, in particular, $v\in \Cmax$), so that  proving the following proposition is sufficient to prove Theorem \ref{geq3}.

\begin{prop}[No hit of zero of exploration]
\label{surv}
Consider $\CMd$ in the critical window for connectivity defined in Condition \ref{cond-crit-conn}. Let $v$ be such that $d_v \geq 3$. Then 
	\begin{equation}
	\P ( \exists t \leq T_{1/2} \colon S_t^{\sss(v)} =0) = o(n^{-1}).
	\end{equation}
\end{prop}
\medskip

Since there are $n$  vertices in the graph, Proposition \ref{surv} indeed proves \eqref{bd-1} in Theorem \ref{geq3}. 
We start the proof with a bound on the depletion of high-degree vertices. 

\begin{lem}[Bound on the depletion of high-degree vertices]
\label{dom}
Consider $\CMd$ in the connectivity critical window defined in Condition \ref{cond-crit-conn} and perform the exploration up to time $T_{1/2}=\max \{t \colon |\Ncal_t| > n/2\}$. Then there exists $\varepsilon >0$ such that
	\begin{equation}
	\P ( \# \{ v \in \Ncal_{T_{1/2}}\colon d_v \geq 3\} < \varepsilon n) = o(n^{-1}).
	\end{equation} 
\end{lem}

\proof Let us consider the exploration from a permutation $\xi$ of the set of the half-edges chosen uniformly at random, as described above \eqref{St-rec}. We call $\mathcal T _{n/2} (\xi)$ the set of vertices such that all their half-edges  are  among the last $n/2$ of the permutation $\xi$. The previous definitions imply that $\mathcal T _{n/2} (\xi) \subseteq \Ncal_{T_{1/2}}$.

We now pick a $k>2$ such that $p_k=\lim n_k/n >0$, from the definition of the connectivity critical window we know that such $k$ exists. We want to find a lower bound on $N_{T_{1/2}}(k)= \# \{ v \in \Ncal_{T_{1/2}}\colon d_v =k\} \geq \# \{ v \in \mathcal T _{n/2} (\xi) \colon d_v =k\}$. 

Before running the exploration, we sequentially locate the half-edges of the vertices of degree $k$ in $\xi$. We stop this process once we have examined $n/(4k)$ vertices, or when we run out of vertices of degree $k$. We define the $\sigma$-algebra ${\mathscr F}^k_i$ generated by the positions of the half-edges of the first $i$ vertices that we have examined. We then find that, at each step $j$, thanks to the stopping conditions, there are still at least $n/4$ available spots among the last $n/2$ half-edges in $\xi$, so that
	\begin{equation}
	\P (v_j \in \mathcal T _{n/2} (\xi) \mid {\mathscr F}^k_{j-1} ) \geq \left( \dfrac{n}{4 \ell_n}\right)^k .
	\end{equation}
We know that $\lim_{n \to \infty} \left( \dfrac{n}{4 \ell_n}\right)^k = \big(1/4d\big)^k\equiv q_k$, so that
	\begin{equation}
	N_{T_{1/2}}(k) \overset{st}{\geq}  {\sf Bin} \Big(\Big(p_k \wedge \dfrac{1}{4k}\Big) n, q_k\Big),
	\end{equation}
	where $\overset{st}{\geq}$ indicates stochastic domination.
By concentration of the binomial distribution (see e.g., \cite{ArrGor89}), there exists a $c=c(a,q_k)$ such that, uniformly in $n$,
	\begin{equation}\label{bin}
	\P \left({\sf Bin} \left(an, q_k\right) \leq \dfrac{an}{2}q_k\right)\leq \e^{-cn} = o(n^{-1}).
	\end{equation}
The claim follows by picking $\varepsilon < \dfrac{1}{2} \Big(p_k \wedge \dfrac{1}{4k}\Big)q_k$. \qed
\bigskip

We notice that $S_{t+1}^{\sss(v)} -S_t ^{\sss(v)}<0$ only when one of the following events occurs:

\begin{itemize}
\item[$\rhd$] $A(t)=\{d_{v(e_2(t))}=1\}$, where $e_2(t)$ is the half-edge to which the $t$th paired half-edge is paired. In this case $S_{t+1}^{\sss(v)} -S_t^{\sss(v)} = -1$. Thanks to Lemma \ref{dom}, if we define ${\mathscr F}_k$ as the $\sigma$-algebra generated by the first $k$ steps of the exploration, then, uniformly for $t \leq T_{1/2}$,
	\eqn{
	\label{A(t)-bd}
	\P (A(t)\mid {\mathscr F}_{t-1}) \leq \dfrac{2\rho_1}{ \sqrt{n}}.
	}
\item[$\rhd$] $B(t)=\{e_2(t) \in \Acal_t\}$, where $e_2(t)$ is the half-edge to which the $t$th paired half-edge is paired. In this case $S_{t+1}^{\sss(v)} -S_t^{\sss(v)} = -2$. From the description of the exploration, we obtain that,
uniformly for $t \leq T_{1/2}$, 
	\eqn{
	\label{B(t)-bd}
	\P (B(t)| {\mathscr F}_{t-1}) \leq \dfrac{S_t^{\sss(v)}-1}{\ell_n -t-1} \leq \dfrac{2 S_t^{\sss(v)}}{n}.
	}
\end{itemize}

Now we prove three lemmas that together will yield Proposition \ref{surv}.

The first lemma contains a lower bound on the survival time of the process. Indeed, we show that \whp{} the component of $v$ is at least of polynomial size with respect to $n$.

\begin{lem}[No early hit of zero] 
\label{1/8}
Let $\CMd$ be in the connectivity critical window defined in Condition \ref{cond-crit-conn}. Then,
	\begin{equation}
	\P ( \exists t \leq n^{1/8}\colon S_t ^{\sss(v)}=0) = o(n^{-1}).
	\end{equation}
\end{lem}

\proof For the process to die out before time $n^{1/8}$ it has to make at least $3$ steps down while it is very low. We can characterize the ways in which this can happen in terms of occurrences of the events $A(t)$ (finding a degree $1$ vertex) and $B(t)$ (finding an already active half-edge).

 We thus need one of the following three events to occur:
	\[\begin{split}
	F_1 =\bigcup_{s_1,s_2,s_3 \leq n^{1/8}}& A(s_1) \cap A(s_2) \cap A (s_3) \cap \{ S_{s_1}^{\sss(v)},S_{s_2}^{\sss(v)},S_{s_3}^{\sss(v)} \leq 3 \}, \\ 
	F_2=\bigcup_{s_1,s_2 \leq n^{1/8}}&  A(s_1) \cap B (s_2) \cap \{ S_{s_1}^{\sss(v)},S_{s_2}^{\sss(v)} \leq 3 \},  \\
	F_3=\bigcup_{s_1,s_2 \leq n^{1/8}}&  B(s_1) \cap B (s_2) \cap \{ S_{s_1}^{\sss(v)},S_{s_2}^{\sss(v)} \leq 4 \}.
	\end{split}
	\]
We estimate using \eqref{A(t)-bd} and \eqref{B(t)-bd} to obtain
	\begin{align}
	\P (F_1) &\leq {{n^{1/8}}\choose{3}} \left(\dfrac{2\rho_1}{ \sqrt{n}}\right)^3 \leq \dfrac{4\rho_1^3}{3} \dfrac{n^{3/8}}{n^{3/2}}= o(n^{-1}), \\
	\P (F_2) &\leq {{n^{1/8}}\choose{2}} \dfrac{2\rho_1}{ \sqrt{n}} \dfrac{6}{n}\leq \dfrac{3 \rho_1}{2}\dfrac{n^{1/4}}{n^{3/2}}= o(n^{-1}),\\
	\P (F_3) &\leq {{n^{1/8}}\choose{2}} \left( \dfrac{8}{n} \right)^2 \leq 32\dfrac{n^{1/4}}{n^{2}}= o(n^{-1}).
	\end{align}
Applying the union bound proves the claim. 
\qed
\medskip

The next lemma proves instead that when the process is sufficiently low, it is very unlikely to decrease further, since we have few active half-edges to create loops with.

\begin{lem}[Unlikely to dip even lower]
\label{down6}
Let  $\CMd$ be in the connectivity critical window  defined in Condition \ref{cond-crit-conn}. Fix $v$ such that $d_v \geq 3$. Then, for every $t \leq T_{1/2}$ and $\gamma > 0$,
	\begin{equation}\label{o-2}
	\P \Big( \sum_{i \leq \gamma n^{1/8}} (S_{t+i+1}^{\sss(v)}-S_{t+1}^{\sss(v)} )\indic_{S_{t+i+1}^{\sss(v)}<S_{t+i}^{\sss(v)}<3\gamma n^{1/8}}\geq 6\Big) = o(n^{-2}).
	\end{equation}
\end{lem}

\proof As in the proof of Lemma \ref{1/8}, we find some events that must occur in order that the event in the left-hand side of \eqref{o-2} occurs. Again we express the possible ways in which the required $6$ steps down can occur in terms of the events $A(t)$ and $B(t)$. We start by introducing some notation. For $1\leq i<j$ and $s_i\geq 0$, we write $A_{[i,j]}(t)=A(t+s_i) \cap \dots \cap A(t+s_j), B_{[i,j]}(t)=B(t+s_i) \cap \dots \cap B(t+s_j)$. Then, for the event in the left-hand side of \eqref{o-2} to occur, we need that one of the following events occurs:
	\[
	\begin{split}
	G_1=\bigcup_{s_1,\ldots,s_6 \leq \gamma n^{1/8}}& A_{[1,6]}(t) \cap \bigcap_{i \leq 6 } \{ S_{t+s_i }^{\sss(v)}\leq 3\gamma n^{1/8}\},\\
	G_2=\bigcup_{s_1,\ldots,s_5 \leq \gamma n^{1/8}}& A_{[1,4]}(t) \cap B(t+s_5) \cap \bigcap_{i \leq 5} \{ S_{t+s_i}^{\sss(v)}\leq 3\gamma n^{1/8} \},\\
	G_3=\bigcup_{s_1,\dots ,s_4 \leq \gamma n^{1/8}}& A_{[1,2]}(t)\cap B_{[3,4]} \cap \bigcap_{i \leq 4}  \{ S_{t+s_i}^{\sss(v)}\leq 3\gamma n^{1/8}\},\\
	G_4=\bigcup_{s_1,\ldots, s_3 \leq \gamma n^{1/8}}& B_{[1,3]}(t)\cap \bigcap_{i\leq 3}  \{ S_{t+s_i}^{\sss(v)}\leq 3\gamma n^{1/8}\}.
	\end{split}
	\]
Again we estimate using \eqref{A(t)-bd} and \eqref{B(t)-bd} to obtain
	\begin{align}
	\P (G_1) \leq &{{\gamma n^{1/8}}\choose{6}}\left(\dfrac{2\rho_1}{ \sqrt{n}}\right)^6 \leq \dfrac{2^6\gamma^6 \rho_1^6}{6!} \dfrac{ n^{6/8}}{n^3}
	= o(n^{-2}),\\
	\P (G_2) \leq &{{\gamma n^{1/8}}\choose{5}}\left(\dfrac{2\rho_1}{ \sqrt{n}}\right)^4 
	\dfrac{6\gamma n^{1/8}}{n} \leq \dfrac{2^53\gamma \rho_1^4\gamma^6}{5!} \dfrac{ n^{6/8}}{n^3}
	= o(n^{-2}),\\
	\P (G_3) \leq &{{\gamma n^{1/8}}\choose{4}}\left(\dfrac{2\rho_1}{ \sqrt{n}}\right)^2 \left( \dfrac{6\gamma n^{1/8}}{n}\right)^2
	\leq \dfrac{2^43^2\gamma^6 \rho_1^2}{4!} \dfrac{ n^{6/8}}{n^3}
	= o(n^{-2}),\\
	\P (G_4) \leq &{{\gamma n^{1/8}}\choose{3}}\left( \dfrac{6\gamma n^{1/8}}{n}\right)^3 \leq \dfrac{6^3\gamma^6 }{3!} \dfrac{ n^{6/8}}{n^3}
	= o(n^{-2}).
	\end{align}
Applying the union bound proves the claim. 
\qed
\bigskip

We now show that not only the exploration survives up to time $n^{1/8}$ but also we have a quite large number of active half-edges.

\begin{lem}[Law of large numbers lower bound on exploration]
\label{many}
Fix $v$ such that $d_v \geq 3$. The exploration on $\CMd$ in the connectivity critical window defined in Condition \ref{cond-crit-conn} satisfies that there exists a $\gamma >0$ such that
	\begin{equation}
	\P (S_{n^{1/8}}^{\sss(v)} < 2\gamma n^{1/8}  ) =o(n^{-1}).
	\end{equation}
\end{lem}

\proof We divide the proof into two cases:
\begin{enumerate}
\item There exists $t < n^{1/8}$ such that $S_t^{\sss(v)} \geq 3\gamma n^{1/8}$. In this case, fix $n$ so large that $3\gamma n^{1/8} -6 \geq 2\gamma n^{1/8}$.
Then, note that in order for $S_{n^{1/8}}^{\sss(v)} < 2\gamma n^{1/8}$ to occur and since $S_{t+1}^{\sss(v)}-S_{t}^{\sss(v)}\geq -2$, we must have that
$\sum_{i \leq \gamma n^{1/8}} (S_{t+i+1}^{\sss(v)}-S_{t+1}^{\sss(v)} )\indic_{\{S_{t+i+1}^{\sss(v)}<S_{t+i}^{\sss(v)}<3\gamma n^{1/8}\}}\geq 6$, which 
by Lemma \ref{down6} implies that $S_{n^{1/8}}^{\sss(v)} \geq 3\gamma n^{1/8} -6 \geq 2\gamma n^{1/8}$ has probability $o(n^{-2})$.

\item  $S_t^{\sss(v)} <3\gamma n^{1/8}$ for all $t \leq n^{1/8}$. In this case, we know from Lemma \ref{down6} that with probability $o(n^{-2})$ the sum of the down steps $(S_{t+i}^{\sss(v)}-S_{t+i+1}^{\sss(v)}) \indic_{\{S_{t+i+1}^{\sss(v)}<S_{t+i}^{\sss(v)}<3\gamma n^{1/8}\}}$ is at most $6$. Under this condition, we recall Lemma \ref{dom} and note that $d_{v_t}\geq 3$ with probability at least $\varepsilon$, for some $\varepsilon >0$, since $n^{1/8} \leq T_{1/2}$. Thus,
	\begin{equation}
	S_{n^{1/8}}^{\sss(v)}  \overset{st}{\geq}{\sf Bin} (n^{1/8}, \varepsilon)-6.
	\end{equation}
By concentration of the binomial distribution (see e.g., \cite{ArrGor89})
	\begin{equation}\label{bin}
	\P \Big({\sf Bin} (n^{1/8}, \varepsilon) \leq \dfrac{1}{2}\varepsilon n^{1/8}\Big)\leq \e^{-cn^{1/8}} = o(n^{-2}).
	\end{equation}
for sufficiently large $n$. The claim now follows by choosing $\gamma < \varepsilon / 4$.
\end{enumerate}
\qed
\bigskip

Now we know that at time $t=n^{1/8}$, with probability $1-o(n^{-1})$, $S_t^{\sss(v)} \geq 2 \gamma n^{1/8}$. This means that from that point onwards, we need at least $\gamma n^{1/8}$ steps for the process to die. Thus, the following lemma is needed to prove Proposition \ref{surv}.

\begin{lem}[Process does not go down too much]
\label{down}
Let  $\CMd$ be in the connectivity critical window defined in Condition \ref{cond-crit-conn}. Fix $v$ such that $d_v \geq 3$. 
Then, for every $\gamma > 0$,
	\begin{equation}
	\P (\exists t \in (n^{1/8},T_{1/2})\colon S_{t+\gamma n^{1/8}}^{\sss(v)}<S_{t}^{\sss(v)}< 3 \gamma n^{1/8}-6) = o(n^{-1}).
	\end{equation}
\end{lem}	
	
\proof
First fix $t\in (n^{1/8},T_{1/2})$. Again we split the proof into two parts:

\begin{enumerate}
\item There exists $i < \gamma n^{1/8}$ such that $S_{t+i}^{\sss(v)} \geq 3\gamma n^{1/8}$. In this case, we again know from Lemma \ref{down6} that $S_{t+\gamma n^{1/8}}^{\sss(v)} \geq 3\gamma n^{1/8} -6 \geq 2\gamma n^{1/8}$ with probability $1-o(n^{-2})$.
\item $S_{t+i}^{\sss(v)} <3\gamma n^{1/8}$ for all $t \leq \gamma n^{1/8}$. In this case we know from Lemma \ref{down6} that with probability $o(n^{-2})$ the sum of the down steps $(S_{t+i}^{\sss(v)}-S_{t+i+1}^{\sss(v)}) \indic_{\{S_{t+i+1}^{\sss(v)}<S_{t+i}^{\sss(v)}<3\gamma n^{1/8}\}}$ is at most $6$. Under this condition we can again write
	\begin{equation}
	S_{t+\gamma n^{1/8}}^{\sss(v)}-S_t ^{\sss(v)} \overset{st}{\geq} \mathrm{Bin} (n^{1/8}, \varepsilon)-6.
	\end{equation}
Formula \eqref{bin} proves that the probability that $S_{t+\gamma n^{1/8}}^{\sss(v)}<S_t ^{\sss(v)}$ is at most $o(n^{-2})$.
\end{enumerate}
The union bound implies that
	\begin{equation}
	\P (\exists t \in (n^{1/8},T_{1/2})\colon S_{t+\gamma n^{1/8}}^{\sss(v)}<S_{t}^{\sss(v)}< 3 \gamma n^{1/8}-6) \leq \ell_n \ o(n^{-2})= o(n^{-1}).
	\end{equation}
\qed
\bigskip

Now we are ready to complete the proof of Proposition \ref{surv}.

\begin{proof}[Proof of Proposition \ref{surv}] Lemmas \ref{1/8} and \ref{many} show that up to time $n^{1/8}$ the process is very unlikely 
to die and very likely to grow at least until polynomial size:
	\begin{equation}
	\P (T_0 > n^{1/8}, S_{n^{1/8}}^{\sss(v)}> 2\gamma n^{1/8}) = 1 - o(n^{-1}).
	\end{equation}
Now we define the sequence of random variables $Q_i = S_{(1+\gamma i)n^{1/8}}^{\sss(v)}$, so that $\P ( Q_0 < 2 \gamma n^{1/8}) = o(n^{-1})$. 
By Lemma \ref{down}, 
	\begin{equation}
	\P\left(\forall i\leq \dfrac{ T_{1/2}}{\gamma n^{1/8}}: \ Q_{i+1} \geq Q_i \right) = 1- o(n^{-1}),
	\end{equation}
and consequently
	\begin{equation}
	\P \left(\forall i \leq \dfrac{ T_{1/2}}{\gamma n^{1/8}}: \ Q_i \geq 2 \gamma n^{1/8}\right) = 1-o(n^{-1}).
	\end{equation}
Since $S_{t+1}^{\sss(v)} -S_t^{\sss(v)} \geq -2$, we know that $S_{t+s}^{\sss(v)} \geq S_t^{\sss(v)} - 2s$, so we conclude that 
	\begin{equation}
	\P (S_t^{\sss(v)} >0\ \forall t \leq T_{1/2} ) = 1 - o(n^{-1}).
	\end{equation}
This completes the proof of Proposition \ref{surv}.
\end{proof}
\bigskip

\noindent
We can now conclude the proof of Theorem \ref{geq3}.
\medskip

\noindent
{\it Proof of Theorem \ref{geq3}.} Proposition \ref{surv} proves \eqref{bd-1} in Theorem \ref{geq3}. To prove \eqref{bd-2} in Theorem \ref{geq3},
we use that if $|\cluster(v)|>n/2$, then $v\in \Cmax$, to bound
	\eqan{
	\E [\#\{ v \in [n] \setminus \Cmax \colon d_v \geq 3\}]
	&\leq \E [\#\{ v\in[n] \colon d_v \geq 3, |\cluster(v)|< n/2\}]=o(1)
	}
by Proposition \ref{surv}. \qed
\bigskip

\noindent
To show that actually the size of the graph without the giant component has bounded expectation we need a slightly stronger result.

\begin{prop}[Clusters of vertices of degree at least three outside $\Cmax$]
\label{out}
Let  $\CMd$ be in the connectivity critical window defined in Condition \ref{cond-crit-conn}. Then
	\begin{equation}
	\E [\# \{ v \notin \Cmax\colon v \leftrightarrow [n]\setminus (\Nscr_1\cup\Nscr_2)\}] \to 0,
	\end{equation}
where, for a set of vertices $A\subseteq [n]$, $v \leftrightarrow A$ denotes that there exists $a\in A$ such that $v$ and $a$ are in the same connected component.
\end{prop}

\proof We have already proved that $\E [\#\{ v \in [n] \setminus \Cmax\colon d_v \geq 3\}] \to 0$.
We now initialize the exploration starting from a vertex $v$ with $d_v \in\{ 2,1\}$. Notice that the probability for the process to survive for $n^{1/8}$ steps without finding vertices of degree $3$ is smaller than $\e^{-cn^{1/8}}$ for some $c>0$, since at every step the probability to find a vertex $w$ with $ d_w\geq 3$ is bounded away from $0$.

\begin{itemize}
\item[$\rhd$] If $d_v =2$ and our exploration finds a vertex $w$ with $ d_w \geq 3$ before time $n^{1/8}$, then for the process to die out before time $n^{1/8}$, we again need one of the events $F_1, F_2, F_3$ to occur. Then we can apply Lemmas \ref{1/8}, \ref{many} and \ref{down} to complete the proof that $\E [\# \{ v \in \Nscr_2 \setminus \Cmax\colon v \leftrightarrow [n]\setminus (\Nscr_1\cup\Nscr_2)\}]\to 0$,  in the same way as in the proof of Proposition \ref{surv}. 
\item[$\rhd$] In the connectivity critical window, we have that $n_1 = O (\sqrt{n})$. If $d_v =1$ and our exploration at a certain point finds a vertex $w$ with $d_w \geq 3$, then for the process to die out before time $n^{1/8}$ we need one of the following two events to occur:
	\eqan{
	F'_1 &=\bigcup_{s_1,s_2 \leq n^{1/8}}A(s_1) \cap A(s_2)  \cap \{ S_{s_1}^{\sss(v)},S_{s_2}^{\sss(v)} \leq 2 \}, \\ 
	F'_2&=\bigcup_{s \leq n^{1/8}} B (s) \cap \{ S_{s}^{\sss(v)} =2 \}.\nonumber
	}
We estimate using \eqref{A(t)-bd} and \eqref{B(t)-bd} to obtain
	\begin{align}
	\P (F'_1) &\leq {{n^{1/8}}\choose{2}} \left(\dfrac{2\rho_1}{ \sqrt{n}}\right)^2 \leq 2\rho_1^2 \cdot \dfrac{n^{1/4}}{n}= o(n^{-1/2}), \\
	\P (F'_2) &\leq n^{1/8}\cdot \dfrac{4}{n}= o(n^{-1/2}).
	\end{align}
Now we can apply Lemmas \ref{many} and \ref{down} to complete the proof that $\E [\# \{ v \in \Nscr_1 \setminus \Cmax: v \leftrightarrow w; d_w \geq 3 \}]\to 0$ in a similar way as in the proof of Proposition \ref{surv}.
\end{itemize}
Since
	\eqan{
	&\E [\# \{ v \notin \Cmax\colon v \leftrightarrow [n]\setminus (\Nscr_1\cup\Nscr_2)\}]\\
	&\qquad= \E [\#\{ v \in [n] \setminus \Cmax \colon d_v \geq 3\}]\nn\\
	&\qquad\qquad+ \E [\# \{ v \in (\Nscr_1 \cup \Nscr_2) \setminus \Cmax\colon v \leftrightarrow [n]\setminus (\Nscr_1\cup\Nscr_2)\}],\nonumber
	}
we obtain the claim. 
\qed

\section{Proof of the Main Theorems }
\label{sec-compl-pf}

We can now finally prove the main theorems, putting together results from the previous two sections.

\begin{proof}[Proof of Theorem \ref{main}]
We know that
	\begin{equation}
	\{ \CMd \text{ is connected} \}= \{C_k(n)=L_k(n)=0 \ \forall k\} \cap \{[n]\setminus  (\Nscr_1 \cup \Nscr_2)\subseteq \Cmax\}.
	\end{equation}
We have proved in Theorem \ref{geq3} that, \whp, $[n]\setminus \Cmax\subseteq \Nscr_1 \cup \Nscr_2$. Thus,
	\eqn{
	\prob(\CMd \text{ is connected})=\prob(C_k(n)=L_k(n)=0 \ \forall k)+o(1).
	}
By Theorem \ref{poicl} and the independence of $C _k, L_k$, for each $j < \infty$,
	\begin{equation}
	\lim_{n \to \infty} \P  (C_k(n)=L_k(n)=0 \ \forall k \leq j) = \prod_{k=1}^{j} \P (C_k =0) \prod_{k=2}^{j} \P (L_k=0).
	\end{equation}
To pass to the limit we use dominated convergence. We compute that
	\begin{equation}
	\E [L_k(n)] = n_1 \dfrac{2 n_2}{\ell_n -1}\dfrac{2 n_2-2}{\ell_n -3}\cdots \dfrac{2 n_2 - 2k +4 }{\ell_n - 2k +3}\dfrac{n_1 -1}{\ell_n -2k+1} 
	\leq \dfrac{n_1^2 (2n_2)^{k-2}}{(\ell_n-2k)^{k-1}}.
	\end{equation}
Since $\dfrac{n_1}{\sqrt{n}}\to \rho_1$, $\dfrac{n_2}{n} \to p_2$ and $\dfrac{\ell_n}{n}\to d$, for each $\varepsilon$ and $n$ big enough such that
	\begin{equation}
	\label{exp-Lk-bd}
	\E [L_k(n)] \leq \dfrac{n_1^2(2 n_2)^{k-2}}{2(\ell_n-2k)^{k-1}} 
	\leq \dfrac{(\rho_1^2+\varepsilon)^2}{2(d-\varepsilon)} \left(\dfrac{2(p_2+\varepsilon)}{d-\varepsilon}\right)^{k-2}.
	\end{equation}
For $\varepsilon$ small enough $2(p_2+\varepsilon)< d-\varepsilon$ so the series on the right hand side of \eqref{exp-Lk-bd} is exponentially small in $k$.
Similarly for $C_k(n)$,
	\begin{equation}
	\E[C_k(n)] =\dfrac{1}{2k} n_2 \dfrac{2n_2 -2}{\ell_n-2}\cdots \dfrac{2n_2 -2k+4}{\ell_n-2k+4}\dfrac{1}{\ell_n-2k+2}\leq \dfrac{(2n_2)^k}{k(\ell_n-2k)^k}.
	\end{equation}
As before, we have for every $\varepsilon>0$ 
	\begin{equation}
	\label{exp-Ck-bd}
	\E [C_k(n)]\leq \dfrac{1}{k}\dfrac{(2n_2)^k}{(\ell_n-2k)^k} 
	\leq \dfrac{(2n_2)^k}{k(\ell_n-2k)^k} \leq \dfrac{(2p_2+2 \varepsilon)^k}{k(d-\varepsilon)^k}.
	\end{equation}
Again, for $\varepsilon>0$ small enough, $2(p_2+\varepsilon)< d-\varepsilon$, so that the series on the right hand side of \eqref{exp-Ck-bd} is exponentially small in $k$.

Since
	\[
	\{C_k(n)=L_k(n)=0 \ \forall k\} = \bigcap_j \{C_k(n)=L_k(n)=0 \ \forall k \leq j\}, 
	\]
we obtain
	\eqan{
	\lim_{n \to \infty} \P  (C_k(n)=L_k(n)=0 \ \forall k) &\leq\lim_{j \to \infty}\lim_{n \to \infty} \P  (C_k(n)=L_k(n)=0 \ \forall k \leq j) \\ 
	&=\prod_{k=1}^{\infty} \P (C_k =0) \prod_{k=2}^{\infty} \P (L_k=0)\nn\\
	&=\exp{\left(-\sum_{k=1}^{\infty}  \dfrac{(2p_2)^k}{2kd^k} -\sum_{k=2}^{\infty} \dfrac{\rho_1^{2} (2p_2)^{k-2}}{2d^{k-1}}\right)}\nn\\
	&= \left(\dfrac{d-2p_2}{d}\right)^{1/2}\exp{\left(-\dfrac{\rho_1^2}{2(d-2p_2)}\right)},\nn
	}
where we use that $-\sum_{k\geq 1} x^k/k=\log(1-x)$ for $x\geq 0$.

For the lower bound, we use
	\eqan{
	\lim_{n \to \infty} \P  (C_k(n)=L_k(n)=0 \ \forall k)  &\geq\lim_{j \to \infty}\lim_{n \to \infty} \P  (C_k(n)=L_k(n)=0 \ \forall k \leq j)\\ 
	&\qquad- \limsup_{j\rightarrow \infty}\limsup_{n\rightarrow \infty} \P( \exists k > j \colon C_k(n)+L_k(n) \geq 1).\nn
	}
We find, using the Markov inequality,
	\eqan{
	\label{truncation}
	&\limsup_{j\rightarrow \infty}\limsup_{n\rightarrow \infty} \P( \exists k > j \colon C_k(n)+L_k(n) \geq 1)\\
	&\qquad= \limsup_{j\rightarrow \infty}\limsup_{n\rightarrow \infty}  \P\Big( \sum_{k > j} ( C_k(n)+L_k(n)) \geq 1\Big)\nonumber\\
	&\qquad\leq \limsup_{j\rightarrow \infty}\limsup_{n\rightarrow \infty} \sum_{k > j} \E [ C_k(n)+L_k(n)] =0,\nonumber
	}
by \eqref{exp-Lk-bd} and \eqref{exp-Ck-bd}.
As a result,
	\begin{equation}
	\label{connfin}
	\lim_{n \to \infty} \P  (C_k(n)=L_k(n)=0 \ \forall k) 
	= \left(\dfrac{d-2p_2}{d}\right)^{1/2}\exp{\left(-\dfrac{\rho_1^2}{2(d-2p_2)}\right)}.
	\end{equation}
From \eqref{connfin}, we obtain \eqref{conn} using Theorem \ref{geq3}.
\medskip

We next investigate the boundary cases in Remark \ref{rem-bc}. The result in \eqref{bc-1} follows in an identical way as in the proof of Theorem \ref{main}.
For the result for $d=\infty$ in \eqref{bc-2}, we notice that if $\E [D_n] \to \infty$ then for, all $k \geq 1$,
	\begin{equation}
	\lim_{n\to \infty} \dfrac{2n_2}{\ell_n-2k} =0,
	\end{equation}
so that $\sum_{k\geq 3} L_k(n) +\sum_{k\geq 1} C_k(n)\convp 0$ by \eqref{exp-Lk-bd} and \eqref{exp-Ck-bd} and the Markov inequality.
Moreover,
	\begin{equation}
	\P (L_2 (n) = 0)=\prod_{i=1}^{n_1} \dfrac{\ell_n - n_1 -i +1}{\ell_n  -2i +1} =  \exp\Big(-\frac{n_1^2}{2 \ell_n} (1+o(1))\Big),
	\end{equation}
so that 
	\begin{equation}
	\lim_{n \to \infty} \P (\CMd \text{ is connected})= \lim_{n \to \infty} \P(L_2(n) =0) = \lim_{n \to \infty} \exp\Big(-\frac{n_1^2}{2 \ell_n}\Big).
	\end{equation}
\medskip

Further, we can explore the distribution of the number of vertices outside the giant. We write it as
	\begin{equation}
	n - |\Cmax | = \sum_{k=1}^\infty k (C_k(n)+L_k(n)) + \# \{ v \notin \Cmax: v \leftrightarrow [n]\setminus  (\Nscr_1 \cup \Nscr_2)\}.
	\end{equation}
From Proposition \ref{out} we know that $\E[\# \{ v \notin \Cmax: v \leftrightarrow  [n]\setminus  (\Nscr_1 \cup \Nscr_2)\}]\to 0$, so that, by the Markov inequality,
	\begin{equation}
	\begin{split}
	n - |\Cmax |&= \sum_{k\geq 1} k (C_k(n)+L_k(n))+\op(1).
	\end{split}
	\end{equation}
By \eqref{exp-Lk-bd} and \eqref{exp-Ck-bd} and dominated convergence, we obtain
	\begin{equation}
	n - |\Cmax |\convd  \sum_{k=1}^\infty k (C_k+L_k),
	\end{equation}
which completes the proof of \eqref{dist}.

Since we have shown convergence of all moments, we also obtain\footnote{In the published version of the article, equation \eqref{eq-exp} had a wrong value for the expected number of vertices in line components in the first version, which has now been corrected.}
	\eqan{\label{eq-exp}
	\lim_{n \to \infty} \E [n - |\Cmax |] 
	&= \lim_{n \to \infty} \sum_{k=1}^\infty k \E[ C_k(n)+L_k(n)]\\
	&=\lim_{j \to \infty} \sum_{k=1}^j k \dfrac{(2p_2)^k}{2kd^k}
	+\sum_{k=2}^j k \dfrac{\rho_1^2(2p_2)^{k-2}}{2d^{k-1}}
	=\dfrac{p_2}{d-2p_2}+\dfrac{\rho_1^2}{d-p_2},\nn
	}
as required.	
\end{proof}

\begin{proof}[Proof of Theorem \ref{mainsim}] If we condition on simplicity, then we already have that $C_1(n)=C_2(n)=0$. Therefore, we find using the same method 
as in the previous proof that
	\eqan{
	&\lim_{n \to \infty} \P  (C_k(n)=L_k(n)=0 \ \forall  k\mid \CMd \text{ is simple}) \\
	&\qquad= \prod_{k=3}^{\infty} \P (C_k =0) \prod_{k=2}^{\infty} \P (L_k=0) 
	=\exp{\left(-\sum_{k=3}^{\infty}  \dfrac{(2p_2)^k}{2kd^k}-\sum_{k=2}^{\infty} \dfrac{\rho_1^{2} (2p_2)^{k-2}}{2d^{k-1}} \right) }\nn\\
	&\qquad=\left(\dfrac{d-2p_2}{d}\right)^{1/2}\exp{\left(-	\dfrac{\rho_1^2}{2(d-2p_2)} +\dfrac{p_2^2+d p_2}{d^2} \right) },\nn
	}
from which we obtain \eqref{conns} thanks to Theorem \ref{geq3}.
\medskip

We recall that $\mathscr{N}_n(\boldsymbol{d})$ denotes the number of simple graphs with degree distribution $\boldsymbol{d}$. We know that
	\begin{equation}
	\mathscr{N}_n(\boldsymbol{d})=\exp\left(  -\dfrac{\nu}{2}- \dfrac{\nu^2}{4} \right) \dfrac{(\ell_n-1)!!}{\prod_{i \in [n]}d_i !}(1+o(1))
	\end{equation}
Since $\CMd$ conditioned on being simple has the uniform distribution over all possible simple graphs with degree sequence $\boldsymbol{d}$,
	\begin{equation}
	\mathscr{N}^C_n(\boldsymbol{d})=\mathscr{N}_n(\boldsymbol{d}) \P(\CMd  \text{ is connected}
	\mid \CMd \text{ is simple}),
	\end{equation}
which yields the claim.
\end{proof}

\section*{Acknowledgments}
The work of LF and RvdH is supported by the Netherlands Organisation for Scientific Research (NWO) through the Gravitation {\sc Networks} grant 024.002.003. The work of RvdH is also supported by the Netherlands Organisation for Scientific Research (NWO) through VICI grant 639.033.806. The authors would like to thank Prof. Onno Boxma for pointing out a computational mistake in \eqref{eq-exp} in the first version of the paper.

\begin{small}
\bibliographystyle{abbrv}

\end{small}

\end{document}